\documentclass{amsart}
\usepackage[T1]{fontenc}
\usepackage{latexsym,amssymb,amsmath,mathrsfs}

\usepackage{amsthm}
\usepackage{longtable}
\usepackage{epsfig}
\usepackage{hhline}
\usepackage{epic}

   \newcommand{\Irr}{\operatorname{Irr}}
   
   \newcommand{\sym}{\mathfrak{S}}

   \newcommand{\Z}{\mathbb{Z}}

   \newcommand{\la}{\lambda}

 \newcommand{\ga}{\gamma}

  \newcommand{\p}{\bar{p}}

  \newcommand{\tis}{s_0}
  \newcommand{\tit}{t_0}

\long\def\symbolfootnote[#1]#2{\begingroup\def\thefootnote{\fnsymbol{footnote}}
\footnote[#1]{#2}\endgroup}

\newtheorem{theorem}{Theorem}[section] 
\newtheorem{corollary}[theorem]{Corollary}

\newtheorem{remark}[theorem]{Remark}

\title[Generalized results on core partitions]
   {On core and bar-core partitions}

\author{Jean-Baptiste Gramain}

\address{Institut de Math\'ematiques de Jussieu\\
Universit\'e Denis Diderot, Paris VII\\
UFR de Math\'ematiques\\
2 place Jussieu\\
F-75251 Paris Cedex 05\\
email: gramain@math.jussieu.fr}

\author{Rishi Nath}

\address{Department of Mathematics and Computer Science\\
York College\\
City University of New York\\
94-20 Guy R. Brewer Blvd\\
Jamaica, NY 11418\\
email: rnath@york.cuny.edu}

\begin{document}
\maketitle
\begin{abstract}
If $s$ and $t$ are relatively prime J. Olsson proved in \cite{Olsson-cores} that the $s$-core of a $t$-core partition is again
a $t$-core partition, and that the $s$-bar-core of a $t$-bar-core partition is again a $t$-bar-core partition.
Here generalized results are proved for partitions and bar-partitions when the restriction that $s$ and $t$ be relatively prime
is removed. 
\end{abstract}

\symbolfootnote[0]{2000 Mathematics Subject Classification 20C30 (primary), 20C15, 20C20 (secondary)}

\section{Introduction}\label{intro}
The basic facts about partitions, hooks and blocks can be found in \cite[Chapter 2]{James-Kerber} or \cite[Chapter 1]{Olsson-Combinatorics}. We recall a few key definitions here. 
A partition $\lambda$ of $n$ is defined as a non-increasing sequence of nonnegative integers $(\lambda_1, \lambda_2,\cdots)$ that sum to $n.$ A partition is represented
graphically by its Young diagram $[\lambda]$, which consists of the set of {\emph{nodes}} $\{(i, \, j) \, | \, (i, \, j)\in \mathbb{N}^2, \,  j\leq\lambda_i\}$. The node $(i, \, j)$ is in the $i$th row and $j$th column of $[\lambda]$. The rows of $[\lambda]$ are labelled from top to bottom, while its columns are labelled from left to right.

To each node $(i, \, j)$ in $[\lambda]$ we associate the {\emph{hook}} $h_{ij}$ of $\lambda$, which consists of the node $(i, \, j)$ itself, together with all the nodes $\{ (i, \, k) \, | \, j <k  \}$ in $[\lambda]$ (i.e. in the same row as and to the right of $(i, \, j)$), and all the nodes $\{ (\ell, \, j) \, | \, i < \ell \}$ (i.e. in the same column as and below $(i, \, j)$). The {\emph{length}} of $h_{ij}$ is the total number of nodes contained in the hook. For any integer $\ell \geq 1$, we call {\emph{$\ell$-hook}} a hook of length $\ell$, and {\emph{$(\ell)$-hook}} a hook of length divisible by $\ell$. The information about the $(\ell)$-hooks in $\lambda$ is encoded in the {\emph{$\ell$-quotient}} $q_{\ell}(\lambda)=(\lambda_0, \, \ldots, \, \lambda_{\ell-1})$ of $\lambda$. The $\lambda_i$'s are partitions whose sizes sum to the number $w$ of $(\ell)$-hooks in $\lambda$ (called the {\emph{$\ell$-weight}} of $\lambda$).

The removal of an $\ell$-hook $h$ in $\lambda$ is obtained by removing the $\ell$ nodes of $[\lambda]$ in $h$, and migrating the disconnected nodes in $[\lambda]$ up and to the left. The result is a partition of $n-\ell$ denoted by $\lambda \setminus h$. By removing all the $(\ell)$-hooks in $\lambda$, one obtains the {\emph{$\ell$-core}} $\gamma_{\ell}(\lambda)$ of $\lambda$. The partition $\gamma_{\ell}(\lambda)$ contains no $(\ell)$-hooks, and is uniquely determined by $\lambda$ (i.e. doesn't depend on the order in which we remove the $\ell$-hooks in $\lambda$). The partition $\lambda$ is entirely determined by its $\ell$-core and $\ell$-quotient.

\smallskip
It is well-known that the irreducible complex characters of the symmetric group $\sym_n$ are labelled by the partitions of $n$. If $p$ is a prime, then the distribution of irreducible characters of $\sym_n$ into $p$-blocks has a combinatorial description known as the Nakayama Conjecture: two characters $\chi_{\la}, \, \chi_{\mu} \in \Irr(\sym_n)$ belong to the same $p$-block if and only if $\la$ and $\mu$ have the same $p$-core (see \cite[Theorem 6.1.21]{James-Kerber}). Hence we define, for each integer $\ell \geq 1$, an {\emph{$\ell$-block of partitions}} of $n$ to be the set of all partitions of $n$ having a common given $\ell$-core.

\medskip
We now recall the analogous notions and results for bar-partitions, which can be found in \cite[Chapter 1]{Olsson-Combinatorics}. A {\emph{bar-partition}} is a partition $\lambda$ comprised of distinct parts. To each bar-partition we associate a shifted Young diagram $S(\lambda)$ obtained by shifting the $i$th row of the usual Young diagram $(i-1)$ positions to the right. The $j$-th node in the $i$-th row will be called the $(i,j)$-node. To each node $(i,j)$ in $S(\lambda)$, one can associate a {\it bar} and {\it bar-length}. For any odd integer $\ell$, a bar-partition $\la$ is entirely determined by its {\emph{$\bar{\ell}$-core}} $\bar{\ga}_{\ell}(\la)$ and its {\emph{$\bar{\ell}$-quotient}} $\bar{q}_{\ell}(\la)$. The {\it bar-core} $\bar{\ga}_{\ell}(\la)$ is obtained by removing from $\la$ all the bars of length divisible by $\ell$ (called $(\ell)$-bars). The {\it bar-quotient} of $\la$ is of the form $\bar{q}_{\ell}(\la)=(\la_0, \, \la_1, \,  \ldots \la_{(\ell-1)/2})$, where $\la_0$ is a bar-partition, $\la_1$, ..., $\la_{(\ell-1)/2}$ are partitions, and the sizes of the $\la_i$'s sum to the number of $(\ell)$-bars in $\la$ (called {\emph{$\bar{\ell}$-weight}} of $\la$).

\smallskip
It is well-known that the bar-partitions of $n$ label the faithful irreducible complex characters of the 2-fold covering group $\tilde{\sym}_n$ of $\sym_n$. These correspond to irreducible projective representations of $\sym_n$, and are known as {\emph{spin-characters}}. If $p$ is an odd prime, then the distribution of spin-characters of $\tilde{\sym}_n$ of positive defect into $p$-blocks has a combinatorial description known as the Morris Conjecture: two spin-characters of $\tilde{\sym}_n$ of positive defect belong to the same $p$-block if and only if the bar-partitions labelling them have the same $\p$-core (see \cite[Theorem 13.1]{Olsson-Combinatorics}).

In analogy with this, we define, for each odd integer $\ell \geq 1$, an {\emph{$\bar{\ell}$-block of partitions}} of $n$ to be the set of all bar-partitions of $n$ having a common given $\bar{\ell}$-core.

\section{Some new results on cores and bar-cores}

In this section, we generalize to arbitrary integers $s$ and $t$ the results on cores and bar-cores proved by J. B. Olsson in \cite{Olsson-cores} when $s$ and $t$ are coprime. 
Note that Olsson's result (\cite[Theorem 1]{Olsson-cores}) was interpreted by M. Fayers through alcove geometry and actions of the affine symmetric group (see \cite{Fayers}). It was also used by F. Garvan and A. Berkovich to bound the number of distinct values their partition statistic (the GBG-rank) can take on a $t$-core (mod $s$) (see \cite[Theorem 1.2]{Berkovich-Garvan}).

We keep the notation as in Section \ref{intro}.

\begin{theorem}\label{stcore}
For any two positive integers $s$ and $t$, the $s$-core of a $t$-core partition is again a $t$-core partition.

\end{theorem}
\begin{remark}
This result was proved by J. B. Olsson in \cite{Olsson-cores}, under the extra hypothesis that $s$ and $t$ are relatively prime. R. Nath then gave in \cite{Nath-cores} a proof of the result in general. We give here another proof which, unlike the one given by Nath, uses Olsson's result, and provides the framework for the proof for bar-partitions.
\end{remark}

\begin{proof}
Consider a $t$-core partition $\la$. Let $g=\gcd(s, \, t)$, and write $s_0=s/g$ and $t_0=t/g$. It's a well-known fact (see e.g. \cite[Theorem 3.3]{Olsson-Combinatorics}) that there is a canonical bijection $\varphi$ between the set of hooks of length divisible by $g$ in $\la$ and the set of hooks in $q_g(\la)=(\la_0, \,  \ldots , \, \la_{g-1})$ (i.e. hooks in each of the $\la_i$'s). For each positive integer $k$ and hook $h$ of length $kg$ in $\la$, the hook $\varphi(h)$ has length $k$. Furthermore, we have $q_g(\la \setminus h) = q_g (\la) \setminus \varphi(h)$.

In particular, since $\la$ is an $t$-core, and since $t = \tit g$, we see that $q_g(\la)$ contains no $\tit$-hook, so that each $\la_i$ is an $\tit$-core.

Now, the $s$-hooks in $\la$ are in bijection with the $\tis$-hooks in $q_g(\la)$. When we remove them all, we obtain that the $s$-core $\ga_s(\la)$ has $g$-core $\ga_g(\ga_s(\la))=\ga_g(\la)$ and $g$-quotient $q_g(\ga_s(\la))=(\ga_{\tis}(\la_0), \,  \ldots \ga_{\tis}(\la_{g-1}))$. But, since $\tis$ and $\tit$ are coprime, the $\tis$-core of each $\tit$-core $\la_i$ is again a $\tit$-core (\cite[Theorem 1]{Olsson-cores}). This shows that $q_g(\ga_s(\la))$ has no $\tit$-hook, which in turn implies that $\ga_s(\la)$ contains no $t$-hook, whence is an $t$-core.

\end{proof}

As we mentionned in Section \ref{intro}, when $p$ is a prime, the study of $p$-cores is linked to that of the $p$-modular representation theory of the symmetric group $\sym_n$ (as they label the $p$-blocks of irreducible characters). When $\ell \geq 2$ is an arbitrary integer, it turns out that it is still possible to describe an $\ell$-modular representation theory of $\sym_n$ (see \cite{KOR}). The theory of $\ell$-blocks obtained in this way is in fact related to the ordinary representation theory of an Iwahori-Hecke algebra of type $\sym_n$, when specialized at an $\ell$-root of unity. K\"ulshammer, Olsson and Robinson proved in \cite{KOR} the following analogue of the Nakayama Conjecture: two characters $\chi_{\la}, \, \chi_{\mu} \in \Irr(\sym_n)$ belong to the same $\ell$-block if and only if $\la$ and $\mu$ have the same $\ell$-core.

It is therefore legimitate to study $\ell$-cores and $\ell$-blocks of partitions. In particular, we obtain from Theorem \ref{stcore} a generalization of \cite[Corollary 3]{Olsson-cores}. We call {\emph{principal}} $\ell$-block of $n$ the $\ell$-block of partitions of $n$ which contains the partition $(n)$ (i.e. the set of partitions labelling the characters of the principal $\ell$-block of $\sym_n$). 

\begin{corollary}\label{cor}
Let $r$, $s$ and $t$ be any positive integers such that $s>r \geq t$, and let $n=as+r$ for some $a \in \Z_{\geq 0}$. Then the principal $s$-block of $n$ contains no $t$-core.

\end{corollary}
\begin{proof}
Suppose the partition $\la$ of $n$ is a $t$-core. The $s$-core $\ga$ of $\la$, which is obtained by removing $s$-hooks, must therefore be a partition of some $m$ which differs from $n$ by a multiple of $s$, i.e. $m=bs+r$ for some $b$ such that $a \geq b \geq 0$. By Theorem \ref{stcore}, $\ga$ is also a $t$-core. Now, if $\la$ was in the principal $s$-block of $n$, then its $s$-core would be the same as that of the cycle $(n)$, hence also a cycle. We would thus have $\ga=(m)$. But since $m \geq r \geq t$, the cycle $(m)$ contains a $t$-hook, hence cannot be a $t$-core.
\end{proof}

In terms of blocks of characters, this means that, if $s$, $t$ and $n$ are as above, then there is no trivial block inclusion of a $t$-block in the principal $s$-block of $\sym_n$ (see \cite{Olsson-Stanton}).

\medskip
We now prove the analogue results for bar-cores, which was proved by Olsson when $s$ and $t$ are odd and coprime (\cite[Theorem 4]{Olsson-cores}).

\begin{theorem}\label{stbarcore}
For any two odd positive integers $s$ and $t$, the $\bar{s}$-core of an $\bar{t}$-core partition is again a $\bar{t}$-core partition.

\end{theorem}

\begin{proof}
Take any $\bar{t}$-core $\la$. Let $g=\gcd(s, \, t)$, and write $s_0=s/g$ and $t_0=t/g$. There is a canonical bijection $\varphi$ between the set of bars of length divisible by $g$ in $\la$ and the set of bars in its $\bar{g}$-quotient $\bar{q}_g(\la)=(\la_0, \, , \, \la_1, \,   \ldots , \, \la_{(g-1)/2})$, where a bar in $\bar{q}_g(\la)$ is either a bar in the bar-partition $\la_0$ or a hook in one of the partitions $\la_1, \,   \ldots , \, \la_{(g-1)/2}$ (see \cite[Theorem 4.3]{Olsson-Combinatorics}). For each positive integer $k$ and bar $b$ of length $kg$ in $\la$, the bar $\varphi(b)$ has length $k$. Furthermore, we have $\bar{q}_g(\la \setminus b) = \bar{q}_g (\la) \setminus \varphi(b)$.

The same argument as in the proof of Theorem \ref{stcore} thus proves that $\la_0$ is an $\bar{\tit}$-core, that each $\la_i$ ($1 \leq i \leq (g-1)/2$) is an $\tit$-core, and that the $\bar{s}$-core $\bar{\ga}_s(\la)$ of $\la$ has $\bar{g}$-quotient $\bar{q}_g(\bar{\ga}_s(\la))=(\bar{\ga}_{\tis}(\la_0), \, \ga_{\tis}(\la_1), \,   \ldots \ga_{\tis}(\la_{(g-1)/2}))$. And, since $\tis$ and $\tit$ are coprime, the $\tis$-core of each $\tit$-core $\la_i$ ($1 \leq i \leq (g-1)/2$) is again a $\tit$-core (\cite[Theorem 1]{Olsson-cores}), and the $\bar{\tis}$-core of the $\bar{\tit}$-core $\la_0$ is again a $\bar{\tit}$-core (\cite[Theorem 4]{Olsson-cores}). This shows that the $\bar{g}$-quotient of $\bar{\ga}_s(\la)$ contains no $\tit$-bar, which finally implies that $\bar{\ga}_s(\la)$ contains no $t$-bar, whence is an $\bar{t}$-core.

\end{proof}

In analogy with the partition case, we call {\emph{principal}} $\bar{\ell}$-block of bar-partitions of $n$ (for $\ell$ odd) the $\bar{\ell}$-block containing the bar-partition $(n)$. Then the same argument as for the proof of Corollary \ref{cor} yields

\begin{corollary}
Let $r$, $s$ and $t$ be any positive integers such that $s$ and $t$ are odd and $s>r \geq t$, and let $n=as+r$ for some $a \in \Z_{\geq 0}$. Then the principal $\bar{s}$-block of $n$ contains no $\bar{t}$-core.

\end{corollary}




\bibliographystyle{plain}
\bibliography{referencesGN}

\end{document}